\documentclass[12pt]{amsart}
\usepackage{epstopdf}
\usepackage[numbers,sort&compress]{natbib}
\bibpunct[, ]{[}{]}{,}{n}{,}{,}
\makeatletter
\def\NAT@def@citea{\def\@citea{\NAT@separator}}
\makeatother
\usepackage{xcolor}

\usepackage{amsmath,amssymb,amsthm,amsfonts,epsfig,color}
\usepackage{hyperref}
\usepackage{verbatim}

\usepackage{tikz}
\usepackage{amsmath}
\usepackage{amssymb}
\usepackage{amsthm}
\usepackage{amsfonts}
\usepackage{arydshln}
\usepackage{enumerate}
\usepackage[thinlines]{easybmat}
\numberwithin{equation}{section}

\newcommand{\oL }{\mathcal{L}(E) }

\newcommand{\z}{{z}}

\newtheorem{theorem}{Theorem}[section]
\newtheorem{proposition}[theorem]{Proposition}
\newtheorem{corollary}[theorem]{Corollary}
\newtheorem{lemma}[theorem]{Lemma}

\newtheorem{example}[theorem]{Example}

\numberwithin{equation}{section}

\title[Binormal block Toeplitz operator]{Binormal block Toeplitz operators with matrix valued circulant symbols}

\author{Nihat Gökhan Göğüş}
\address{Faculty of Engineering and Natural Sciences, Sabanci University Istanbul, Turkey}
\email{gokhan.gogussabanciuniv.edu}

\author{Rewayat Khan}
\address{Faculty of Engineering and Natural Sciences, Sabanci University Istanbul, Turkey}
\email{rewayat.khan@sabanciuniv.edu}

\author{Eungil Ko}
\address{Department of Mathematics, Ewha Womans University, Seoul 03760, Korea}
\email{eiko@ewha.ac.kr}

\author{Ji Eun Lee}
\address{Department of Mathematics and Statistics, Sejong University, Seoul, 05006, Republic of Korea}
\email{jieunlee7@sejong.ac.kr}

\begin{document}
\begin{abstract} {
This paper focuses on the binormality of block Toeplitz operators  with matrix valued circulant symbols.
We also study some $\Gamma$-dilations of Toeplitz operators.
Moreover, we also analyze the invariant subspace of Toeplitz operators with matrix-valued symbols.}
\end{abstract}
$\\[-7ex]$
\maketitle
\section{\bf{Introduction}}
Let $E$ be a  separable complex Hilbert space and let ${\mathcal{L}}({E})$ be the algebra of all bounded linear operators on $E$. For an operator $T\in {\mathcal{L}}(E)$ $T^{*}$ denote the adjoint of $T$. For $S,T\in {\mathcal{L}}(E)$,
set $[S,T]=ST-TS$. An
operator $T\in{\mathcal{L}}(E)$ is said to be \textit{self-adjoint} if
$T=T^{\ast}$, \textit{unitary} if $T^{\ast}T=TT^{\ast}=I$,
 \textit{normal} if $[T^{\ast},T]=0$,  \textit{quasinormal} if $[T^{\ast}T,T]=0$, and  \textit{binormal} if $[T^{\ast}T,TT^{\ast}]=0$,
 respectively. An operator $T\in {\mathcal{L}}(E)$ is called {\it subnormal} if $T$ has a normal extension $N$, i.e., there is a Hilbert space $F$ containing $E$  and a normal operator $N\in \mathcal{L}(F)$ such that $E$ is invariant under $N$, i.e.,
 $NE\subseteq E$ and $T=N\mid_{E}$.

 Let  $\mathbb{R}$
(resp., $\mathbb{C}$) for the set of real (resp., complex) numbers. Let $L^{2}(\mathbb{T})$ be the set of all measurable functions on the unit circle $\mathbb{T}=\partial \mathbb{D}$ whose Fourier coefficients are square summable.
Let $H^2$ be the classical Hardy space in the unit disk $\mathbb{D}=\{\lambda\in\mathbb{C}:|\lambda|<1\}$. Then $H^{2}$ can be thought of as a closed subspace of $L^{2}(\mathbb{T})$ of the normalized Lebesgue measure on $\mathbb{T}$ whose negative Fourier coefficients vanish. The space of essentially bounded functions in $L^2({\mathbb{T}})$ is denoted by $L^{\infty}$, and the bounded analytic functions by $H^{\infty}$.

\medskip

The circulant matrices are Toeplitz matrices which are of the form
\begin{eqnarray*}
T=(a_{i-j})_{i,j=0}^{n-1}=circ(a_{0}, a_{1},\cdots,a_{n-1})=
\begin{bmatrix}
a_{0} & a_{1} &a_{2} &\cdots & a_{n-1} \\
a_{n-1} & a_{0} & a_{1}&\cdots & a_{n-2} \\
a_{n-2} & a_{n-1} &a_{0}& \cdots & a_{n-3} \\
\vdots & \vdots & \vdots &\ddots & \vdots \\
a_{1} & a_{2} &a_{3} &\cdots & a_{0}
\end{bmatrix}.
\end{eqnarray*}
It is a commutative subalgebra of $n\times n$ Toeplitz matrices  denoted by $\mathcal{T}_{n}$ (see \cite{SH}).
\medskip

Let $S, T\in \mathcal{L}(E)$. Then $S$ and $T$ are said to be {\it unitarily equivalent} if there exists a unitary operator $U\in \mathcal{L}(E)$ such that $S=U^{*}TU$.
Let $\mathcal{M}$ be a non-trivial closed subspace of $E$. Then we say that $\mathcal{M}$ is an invariant subspace of $T\in\mathcal{L}(E)$ if $T\mathcal{M}\subset\mathcal{M}$.
The subspace $\mathcal{M}$ reduces the operator $T$ if both $\mathcal{M}$ and $\mathcal{M}^{\perp}$ are invariant under $T$.

\begin{theorem}{\em \cite[Exercise 1.10.2, P. 58]{FM}}\label{DECOMP}
Let $T\in\mathcal{L}(E)$ and let $\mathcal{M}$ be a non-trivial closed subspace of $E$. Then the matrix representation of $T$ with respect to the decomposition $E=\mathcal{M}\oplus \mathcal{M}^{\perp}$ is block diagonal if and only if the subspace $\mathcal{M}$ is reducing for $T$.
\end{theorem}

\smallskip

{
This paper is structured as follows. Section 2 provides a brief review of vector-valued analytic function spaces and their operators, which are essential for our subsequent analysis. In Section 3, we discuss properties of (binormal) Toeplitz operators with matrix-valued circulant symbols. Section 4 defines $\Gamma$-dilation and presents a proof that a block Toeplitz operator with a Toeplitz matrix symbol has a reducing subspace.
 We also include a discussion on the binormality of these operators.}


\section{\bf{Preliminaries}}
{

Let $E$ be a complex separable Hilbert space. In what follows $\|\cdot\|_E$ and $\langle\cdot,\cdot\rangle_{E}$ will denote the norm and the inner product in $E$, respectively. The space $L^{2}(E)$ consists of functions $f\colon\mathbb{T}\to E$ such that $f$ is measurable and $$\int_{\mathbb{T}} \| f(\z)\|_E^2\,dm(\z)<\infty$$
where $m$ is the normalized Lebesgue measure on $\mathbb{T}$. The space $L^2(E)$ is a Hilbert space with the inner product given by
\begin{equation*}
\langle f,g\rangle_{L^{2}(E)}=\int_{\mathbb{T}}\langle f(\z),g(\z)\rangle_{E}\,dm(\z),\quad f,g\in L^2(E).
\end{equation*}

Equivalently, $L^2(E)$ consists of elements $f:\mathbb{T}\to E$ of the form
\begin{equation}\label{1}
 \begin{array}{rl}
 f(\z)=\sum\limits_{n=-\infty}^{\infty}a_{n}\z^n
  \text{ such that}  \sum\limits_{n=-\infty}^{\infty}\|a_{n}\|_{E}^{2}<\infty\end{array}\end{equation}
with $\{a_{n}\}\subset E$.

If $f\in L^2(E)$ is given by \eqref{1}, then its Fourier series converges in the  $L^2(E)$ norm and
$$\|f\|_{L^2(E)}^2=\int_{\mathbb{T}} \| f(\z)\|_{E}^2\,dm(\z)=\sum\limits_{n=-\infty}^{\infty}\|a_{n}\|_{E}^{2}.$$
Moreover, for $\displaystyle g(\z)=\sum_{n=-\infty}^{\infty}b_{n}\z^n\in L^2(E)$ we have
$$\langle f, g\rangle_{L^2(E)}=\sum\limits_{n=-\infty}^\infty \langle a_n,b_n\rangle_{E}=\int_{\mathbb{T}}\langle f(\z),g(\z)\rangle_{E}\,dm(\z).$$

The vector valued Hardy space $H^2(E)$ is defined as the set of all the elements of $L^2(E)$ whose Fourier coefficients with negative indices vanish, that is,
$$H^2(E)=\left\{f\in L^2(E):  a_n=0, n\leq-1\right\}.$$
Each $f\in H^2(E)$, $\displaystyle f(\z)=\sum_{n=0}^{\infty} a_n \z^n$, can also be identified with a function
$$f(\lambda)=\sum\limits_{n=0}^\infty a_n\lambda^n,\quad \lambda\in\mathbb{D},$$
analytic in the unit disk $\mathbb{D}$ (the boundary values $f(\z)$ can be obtained from the radial limits, which converge to the boundary function in the $L^2(E)$ norm). Denote by $P$ the orthogonal projection $P:L^2(E)\to H^2(E)$.

 The space of essentially bounded functions in $L^{2}(E)$ is denoted by
 $L^{\infty}(E)$ and bounded functions on $\mathbb{D}$ in $H^{2}(E)$ is denoted by $H^{\infty}(E)$.

Now let $\oL $ be the algebra of all bounded linear operators on $E$ equipped with the operator norm $\|\cdot\|_{\oL}$. We can define $\oL$-valued, i.e., operator valued functions. We denote these spaces by $L^{2}(\oL)$ and $H^{2}(\oL)$, respectively. The space of operator valued, essentially bounded functions on $\mathbb{T}$ is denoted by $L^{\infty}(\oL)$, and the space of bounded analytic functions in $H^{2}(\oL)$ is denoted by $H^{\infty}(\oL)$.

Each $\Phi\in L^{\infty}(\oL)$ admits a formal Fourier expansion (a.e. on $\mathbb{T}$)
	\begin{equation}\label{jeden}
\mathbf{\Phi}(\z)=\sum_{n=-\infty}^{\infty}{\Phi}_n \z^n\quad\text{with }\{\Phi_n\}\subset  \oL
\end{equation}
defined by
	\begin{equation}\label{fourier2}
\Phi_nx=\int_\mathbb{T} \overline{\z}^n \mathbf{\Phi}(\z)x\,dm(\z) \quad\text{for }  x\in E
\end{equation}
(integrated in the strong sense). Let
$$H^{2}(\oL)=\left\{\Phi\in L^2(\oL): \Phi_{n}=0, n\leq -1\right\}.$$

Each bounded analytic $\Phi$ is of the form
\begin{equation}
\label{F1}\Phi(\lambda)=\sum_{n=0}^{\infty}{\Phi}_n\lambda^n,\quad \lambda\in\mathbb{D},
\end{equation}
and can be identified with the boundary function
\begin{equation}\label{F2}
\Phi(\z)=\sum_{n=0}^{\infty}{\Phi}_n\z^n\in L^{\infty}(\oL).
\end{equation}
Conversely, each $\Phi\in L^{\infty}(\oL)$ given by \eqref{F2} can be extended by \eqref{F1} to a function bounded and analytic in $\mathbb{D}$. In each case the coefficients $\{\Phi_n\}$ can be obtained by \eqref{fourier2} and the norms $\|\cdot\|_{\infty}$ of the boundary function and its extension coincide (see \cite[p. 232]{berc}).

\medskip

 We consider $\oL $ as a Hilbert space with the Hilbert--Schmidt norm and we may also define the spaces $L^2(\oL)$ and $H^2(\oL)$ as above.
Since here the Hilbert--Schmidt norm and the operator norm are equivalent, we have
$$L^{\infty}(\oL )\subset L^2(\oL ),\quad H^{\infty}(\oL )\subset H^2(\oL ).$$
Moreover, it is not difficult to verify that if $\Phi\in L^2(\oL)$ is given by
$$\Phi(z)=\sum_{n=-\infty}^{\infty}\Phi_{n}z^n, \quad \Phi_n\in \oL,$$
where the series is convergent in the $L^2(\oL)$-norm, then
$$\Phi^*(z)=[\Phi(z)]^*=\sum_{n=-\infty}^{\infty}(\Phi_{-n})^{*}z^n.$$
We thus have
$$L^2(\oL)=\left[zH^2(\oL)\right]^*\oplus H^2(\oL).$$

To each $\Phi\in L^\infty(\oL)$ there corresponds a multiplication operator $M_\Phi:L^2(E)\to L^2(E)$: for $f\in L^2(E)$,
$$(M_\Phi f)(\z)=\Phi(\z)f(\z)\quad \text{a.e. on }\mathbb{T}.$$
By $T_{\Phi}$ we will denote the compression of $M_{\Phi}$ to the Hardy space $H^{2}(E)$: $T_{\Phi}:H^2(E)\to H^2(E)$, $$T_{\Phi}f=PM_{\Phi}f\quad\text{for }f\in H^2(E).$$

For $\Phi\in L^{\infty}(\oL )$ the operators $M_{\Phi}$ and $T_{\Phi}$ can be  densely defined, on $L^{2}(E)$ and $H^{2}(E)$, respectively.
 For more details on spaces of vector valued functions we refer the reader to \cite{berc,NF,RR}.

}

\medskip

In particular, if a matrix-valued function $\Phi$ has the following representation; $\Phi=\begin{pmatrix} \varphi_1 & \varphi_2\cr \varphi_3 &\varphi_4\end{pmatrix}$, then  the block Toeplitz operator $T_{\Phi}$ has the following representation; $$T_{\Phi}=\begin{pmatrix} T_{\varphi_1} & T_{\varphi_2}\cr T_{\varphi_3} & T_{\varphi_4}\end{pmatrix}.$$
If $\Phi\in H^{\infty}(\mathcal{L}(E))$, then $T_{\Phi}f=M_{\Phi}f$, where $M_{\Phi}$ is the multiplication operator on $H^{2}(E)$. The operator $S=T_{zI_{n}}$ is an example of a block Toeplitz operator. It is called a shift operator. Toeplitz operator $T_{\Phi}$ is called an {\it analytic} Toeplitz operator if $\Phi\in H^{2}(\mathcal{L}(E))$, and a {\it coanalytic} if $\Phi^{*}\in H^{2}(\mathcal{L}(E))$.

For $\Phi\in L^{\infty}(\mathcal{L}(E))$ we write $$\Phi=[z\Phi_{-}]^{*}+\Phi_{+}, \quad \text{where}~~ \Phi_{+}, \Phi_{-}\in H^{2}(\mathcal{L}(E)).$$

A function $\Theta\in H^{\infty}(\mathcal{L}(E))$ is called an \emph{inner function} if $\Theta(z)^{*}\Theta(z)=I_{E}$ a.e. on $\mathbb{T}$.

\medskip

\textbf{Beurling-Lax Theorem}. A nontrivial subspace $M$ of $H^{2}(\mathcal{L}(E))$ is $S=T_{zI}$-invariant if and only if there exists an inner function $\Theta\in H^{\infty}(\mathcal{L}(E))$ such that $M=\Theta H^{2}(E)$.

We recall that a function $\varphi\in L^{\infty}$ is said to be of \emph{bounded type} if there are analytic functions
$\varphi_{1}, \varphi_{2}\in H^{\infty}$ such that
$$\varphi(z)=\frac{\varphi_{1}(z)}{\varphi_{2}(z)}\quad \text {for almost all} ~ z\in \mathbb{T}.$$

For an operator valued function $\Phi=[\varphi_{ij}]\in L^{\infty}(\mathcal{L}(E))$, we say that $\Phi$ is of \emph{bounded type} if every $\varphi_{ij}$ is of bounded type and $\Phi$ is \emph{rational} if each entry $\varphi_{ij}$ is a rational function. A matrix valued trigonometric polynomial of $\Phi$ is a representation of the form
$$\Phi(z)=\sum_{n=-N}^{N}\Phi_{n}z^{n}.$$

\section{\bf{Binormal block Toeplitz operators with matrix valued circulant symbols}}


The following lemma  gives the relation between the orthogonal projections and  a unitary operator. Moreover, it is elementary, but it will be useful throughout our paper.
\begin{lemma}\label{TauP}
Let $E$ be a Hilbert space and $\mathcal{M}$ be a closed subspace of $E$. Let $P$ denote the orthogonal projection from $E$ onto $\mathcal{M}$. If $\tau:E\to E$ is a unitary operator and $Q$ denotes the orthogonal projection from $\tau(E)$ onto $\tau (\mathcal{M})$, then
$$\tau P=Q\tau.$$
\end{lemma}

\begin{proof}
Let $f\in E$. Then $f=f_{1}+f_{2}$ where $f_{1}\in \mathcal{M}$ and $f_{2}\in \mathcal{M}^{\perp}$. Thus we have
$$Pf=f_{1}$$
and hence $$\tau Pf=\tau f_{1}. $$
Therefore,  $\tau f_{1}\in \tau (\mathcal{M})$, $\tau f_{2}\in \tau (\mathcal{M}^{\perp})$, and $\tau f_{1}\perp\tau f_{2}$. Therefore, $\tau f$ can also be written uniquely as follows
$$\tau f=\tau f_{1}+\tau f_{2}.$$ Hence we get that $$Q\tau f=\tau f_{1}=\tau Pf.$$
\end{proof}

Let us remind the definition of the circulant matrices, i.e., an $n\times n$ Toeplitz matrix of the form
$$C=(a_{i-j})_{i,j=0}^{n-1}=circ(a_{0}, a_{1},\cdots,a_{n-1}){~\mbox{for}~} a_{i}\in \mathbb{C}.$$
Let $\mathcal{C}_{n}$ be the space of all $n \times n$ circulant matrices and let $\mathcal{T}_{n}$ be the space of all Toeplitz
matrices. Then $\mathcal{C}_{n}\subset\mathcal{T}_{n}\subset M_{n}$, and is a commutative subspace of all Toeplitz matrices. Moreover, it is a maximal commutative subalgebra of $M_{n}$. Then $\mathcal{C}_{n}$ is closed under the adjoint (or conjugate transpose) operation. It is a commutative subalgebra of $n\times n$ Toeplitz matrices (see \cite{SH}).
\begin{lemma}{\em \cite[Lemma 3.2]{MAK}}
The space $\mathcal{C}_{n}$ is inverse closed.
\end{lemma}

In this section, we  study Toeplitz operators $T_{\Phi}$ such that $\Phi\in L^{\infty}(\mathcal{C}_{n})$. The series representation of $\Phi\in L^{\infty}(\mathcal{C}_{n})$ is given by
\begin{equation}\label{E_0}
\Phi(z)=\sum_{n=-\infty}^{\infty}\Phi_{n}z^{n} \quad \text{for}~~\Phi_{n}\in \mathcal{C}_{n}.
\end{equation}

Specially, for $n=2$,  let $\varphi_{0},\varphi_{1}\in L^{\infty}(\mathbb{T})$ be given by
$$\varphi_{0}(z)=\sum_{-\infty}^{\infty}a_{n}z^{n}~\mbox{and}~ \varphi_{1}(z)=\sum_{-\infty}^{\infty}b_{n}z^{n}$$ and
$$
\Phi(z)=
\begin{bmatrix}
\varphi_{0}(z)& \varphi_{1}(z)\\
\varphi_{1}(z)& \varphi_{0}(z)
\end{bmatrix}
\in L^{\infty}(\mathcal{C}_{2}).$$
Then
\begin{align*}
\Phi(z)
&=\begin{bmatrix}
\varphi_{0}(z)& \varphi_{1}(z)\\
\varphi_{1}(z)& \varphi_{0}(z)
\end{bmatrix}\\
&=\begin{bmatrix}
\sum_{-\infty}^{\infty}a_{n}z^{n}&\sum_{-\infty}^{\infty}b_{n}z^{n}\\
\sum_{-\infty}^{\infty}b_{n}z^{n}& \sum_{-\infty}^{\infty}a_{n}z^{n}
\end{bmatrix}\\
&=
\begin{bmatrix}
...+a_{-1}\bar z +a_{0}+a_{1}z+...& ...+b_{-1}\bar z +b_{0}+b_{1}z+...\\
...+b_{-1}\bar z +b_{0}+b_{1}z+...& ...+b_{-1}\bar z +b_{0}+b_{1}z+...
\end{bmatrix}\\
&=...+\begin{bmatrix}
a_{-1}& b_{-1}\\
b_{-1}& a_{-1}
\end{bmatrix}\bar z+\begin{bmatrix}
a_{0}& b_{0}\\
b_{0}& a_{0}
\end{bmatrix}+
\begin{bmatrix}
a_{1}& b_{1}\\
b_{1}& a_{1}
\end{bmatrix}z+...\\
&=...+\Phi_{-1}\bar z+\Phi_{0}+\Phi_{1}z+...
\end{align*}
where
$\Phi_{i}\in \mathcal{C}_{2}$ are constant circulant matrices for $i\in{\mathbb Z}$. Hence (\ref{E_0}) holds.

\medskip

\begin{lemma}\label{Diag}
The class $\mathcal{C}_{n}$ of circulant matrices is simultaneously diagonalizable, that is, for every $C\in\mathcal{C}_{n}$ there exists a unitary matrix $U$ such that $$U^{*}CU=\Lambda,$$ where $U=(v_{k})_{k=0}^{n-1}=(v_0,\cdots,v_{n-1})$ is an $n\times n$ matrix and $\Lambda$ is a diagonal matrix having diagonal entries $\lambda_{0}, \lambda_{1},\cdots,\lambda_{n-1}$ {\em(}given in \eqref{Eig}{\em)} which are the eigenvalues of $C$.
\end{lemma}

\begin{proof}
If $C$ is a circulant matrix in $\mathcal{C}_{n}$, then the eigenvalues of $C$ are given by
\begin{eqnarray}\label{Eig}
\lambda_{k}=\sum_{j=0}^{n-1}a_{j}\mu_{n}^{jk}=a_0\mu_n^{0k}+\cdots+a_{n-1}\mu_{n}^{(n-1)k}
\end{eqnarray}
where $\mu_{n}=e^\frac{2\pi i}{n}$ is the $n$-th root of unity, and $k=0,1,\cdots,n-1$. Then eigenvectors $v_k$ corresponding to the eigenvalues $\lambda_{k}$ are given by
$$v_{k}=\frac{1}{\sqrt{n}}(1, \mu_{n}^{k}, \mu_{n}^{2k},\cdots,\mu_{n}^{(n-1)k})^{T}.$$
Since the eigenvectors corresponding to distinct eigenvalues are orthogonal, we have this result.
\end{proof}

\medskip

 Remark that circulant matrices on ${\mathbb C}^n$ are  normal matrices, in general.
If $C=U\Lambda U^{\ast}$, then $$C^{\ast}C=U\Lambda^{\ast}\Lambda U^{\ast}=U\Lambda \Lambda^{\ast}U^{\ast}=CC^{\ast}.$$ Thus $C$ is normal. In Lemma 3.3, the following matrix $\Phi$ is not normal, in general.

\begin{lemma}\label{UE}
Let $\Phi\in L^{\infty}(\mathcal{C}_{n})$, i.e.,
\begin{align*}
\Phi(z)
&=(\varphi_{i-j}(z))_{i,j=0}^{n-1}=circ(\varphi_{0}(z), \varphi_{1}(z),\cdots,\varphi_{n-1}(z))\\
&=\begin{bmatrix}
\varphi_{0}(z) & \varphi_{1}(z) &\varphi_{2}(z) &\cdots & \varphi_{n-1}(z) \\
\varphi_{n-1}(z) & \varphi_{0}(z) & \varphi_{1}(z)&\cdots & \varphi_{n-2}(z) \\
\varphi_{n-2}(z) & \varphi_{n-1}(z) &\varphi_{0}(z)& \cdots & \varphi_{n-3}(z) \\
\vdots & \vdots & \vdots &\ddots & \vdots \\
\varphi_{1}(z) & \varphi_{2}(z) &\varphi_{3}(z) &\cdots & \varphi_{0}(z)
\end{bmatrix}
\end{align*}
Then $\Phi$ is unitarily equivalent to a diagonal matrix $\Lambda$.
\end{lemma}

\begin{proof}
Since $\Phi(z)\in L^{\infty}(\mathcal{C}_{n})$, $$\Phi(z)=\sum_{k=-\infty}^{\infty}\Phi_{k}z^{k} \quad \text{for}~~\Phi_{k}\in \mathcal{C}_{n}.$$
If $U$ is a constant unitary matrix as in  Lemma \ref{Diag}, then
\begin{eqnarray*}\label{Mat00}
U^{*}\Phi(z)U&=&U^{\ast}(\sum_{k=-\infty}^{\infty}\Phi_{k}z^{k})U \cr
&=&\sum_{k=-\infty}^{\infty}U^{\ast}\Phi_{k}Uz^{k} \cr
&=&\sum_{k=-\infty}^{\infty}\Lambda_{k}z^k\cr
&=&\begin{bmatrix}
\sum_{k=-\infty}^{\infty}\lambda_{k,0}z^k & 0 &0 &\cdots & 0 \\
0 & \sum_{k=-\infty}^{\infty}\lambda_{k,1}z^k  & 0&\cdots & 0 \\
0 & 0 &\dots& \cdots & 0 \\
\vdots & \vdots & \vdots &\ddots & \vdots \\
0 & 0 &0 &\cdots & \sum_{k=-\infty}^{\infty}\lambda_{k,n-1}z^k
\end{bmatrix}\cr
&=&\begin{bmatrix}
\lambda_{0}(z) & 0 &0 &\cdots & 0 \\
0 & \lambda_{1}(z) & 0&\cdots & 0 \\
0 & 0 &\lambda_{2}(z)& \cdots & 0 \\
\vdots & \vdots & \vdots &\ddots & \vdots \\
0 & 0 &0 &\cdots & \lambda_{n-1}(z)
\end{bmatrix}
=\Lambda(z)
\end{eqnarray*}
where $\Lambda_k=\begin{bmatrix}
\lambda_{k,0} & 0 &0 &\cdots & 0 \\
0 & \lambda_{k,1} & 0&\cdots & 0 \\
0 & 0 &\lambda_{k,2}& \cdots & 0 \\
\vdots & \vdots & \vdots &\ddots & \vdots \\
0 & 0 &0 &\cdots & \lambda_{k, n-1}
\end{bmatrix}$ and $\{\lambda_{k,0},\cdots,\lambda_{k,n-1}\}$ are eigenvalues of $\Phi_{k}$ as in the proof of Lemma 3.2.
Therefore,  $\Phi(z)$ is unitarily equivalent to a diagonal matrix $\Lambda(z)$. Since $U$ is a constant unitatry matrix,
it follows that
$$(U^{\ast}\Phi U)(z)=U^{\ast}\Phi(z) U
=\Lambda(z).$$
\end{proof}


\medskip

\begin{theorem}\label{BTOP}
Let $\Phi\in L^{\infty}(\mathcal{C}_{n})$ such that $U^{*}\Phi(z)U=\Lambda(z)$ as in Lemma \ref{UE}. Then the following statements hold.\\
{\em(i)} $T_{\Phi}$ is unitarily equivalent to $T_{\Lambda}$.\\
{\em(ii)} $T_{\Phi}$ is binormal if and only if $T_{\Lambda}$ is binormal
 where $$\Lambda(z)=diag(\lambda_{0}(z), \lambda_{1}(z),\cdots,\lambda_{n-1}(z)).$$
\end{theorem}

\begin{proof} (i)
Let $\Phi\in L^{\infty}(\mathcal{C}_{n})$. Then by Lemma \ref{UE} there exists a constant unitary matrix $U$ such that $U^{*}\Phi U=\Lambda$.
Thus, for $f\in H^{2}(E)$,
\begin{align*}
T_{\Lambda}f
&=T_{U^{*}\Phi U}f=P_{H^{2}(E)}(U^{*}\Phi U f).
\end{align*}
{Since  $U\in \mathcal{L}(E)$ is a constant unitary operator, it follows from Lemma \ref{TauP} that},
$$P_{H^{2}(E)}U^{*}=U^{*}P_{UH^{2}(E)}=U^{*}P_{H^{2}(E)},$$ Therefore we have
\begin{align*}
T_{\Lambda}f
&=P_{H^{2}(E)}(U^{*}\Phi U f)\\
&=U^{*}P_{UH^{2}(E)}(\Phi Uf)\\
&=U^{*}P_{H^{2}(E)}(\Phi Uf)\\
&=U^{*}T_{\Phi}( Uf)\\
&=U^{*}T_{\Phi}U(f).
\end{align*}
Hence $T_{\Phi}$ is unitarily equivalent to $T_{\Lambda}$.

(ii) Since $\Phi\in L^{\infty}(\mathcal{C}_{n})$,  we have $$U^{*}\Phi(z)U=\Lambda(z)=diag(\lambda_{0}(z), \lambda_{1}(z),\cdots,\lambda_{n-1}(z)).$$ Then
$$T_{\Lambda}=diag(T_{\lambda_{0}}, T_{\lambda_{1}},\cdots,T_{\lambda_{n-1}})$$
and
$$T^{*}_{\Lambda}=diag(T^{*}_{\lambda_{0}}, T^{*}_{\lambda_{1}},\cdots,T^{*}_{\lambda_{n-1}}).$$
The product of two diagonal operators gives us
$$T^{*}_{\Lambda}T_{\Lambda}=diag(T^{*}_{\lambda_{0}}T_{\lambda_{0}}, T^{*}_{\lambda_{1}}T_{\lambda_{1}},\cdots,T^{*}_{\lambda_{n-1}}T_{\lambda_{n-1}})$$ and

$$T_{\Lambda}T^{*}_{\Lambda}=diag(T_{\lambda_{0}}T^{*}_{\lambda_{0}}, T_{\lambda_{1}}T^{*}_{\lambda_{1}},\cdots,T_{\lambda_{n-1}}T^{*}_{\lambda_{n-1}}).$$
Hence we have
\begin{equation}\label{eq1}
T^{*}_{\Lambda}T_{\Lambda}T_{\Lambda}T^{*}_{\Lambda}
=diag(T^{*}_{\lambda_{0}}T_{\lambda_{0}}T_{\lambda_{0}}T^{*}_{\lambda_{0}}, \cdots,T^{*}_{\lambda_{n-1}}T_{\lambda_{n-1}}T_{\lambda_{n-1}}T^{*}_{\lambda_{n-1}})
\end{equation}
and
\begin{equation}\label{eq2}
T_{\Lambda}T^{*}_{\Lambda}T^{*}_{\Lambda}T_{\Lambda}
=diag(T_{\lambda_{0}}T^{*}_{\lambda_{0}}T^{*}_{\lambda_{0}}T_{\lambda_{0}} ,\cdots,T_{\lambda_{n-1}}T^{*}_{\lambda_{n-1}}T^{*}_{\lambda_{n-1}}T_{\lambda_{n-1}}).
\end{equation}
From (\ref{eq1}) and (\ref{eq2}) we have $T_{\Lambda}$ is binormal if and if $T_{\lambda_{0}}, T_{\lambda_{1}},\cdots,T_{\lambda_{n-1}}$ are binormal. By (i), we have that $T_{\Phi}$  is unitarily equivalent to $T_{\Lambda}$.
Since the unitary equivalent relation preserves the binormality, we conclude that $T_{\Phi}$  is binormal if and only if  $T_{\Lambda}$ is binormal.
\end{proof}

\medskip

\begin{corollary}
Let $\Phi(z)=circ(\varphi_0(z), \varphi_1(z))$. Then $T_{\Phi}$ is unitarily equivalent to $T_{\Lambda}$ where $\Lambda(z)=\begin{bmatrix} \varphi_0(z)+\varphi_1(z) & 0 \cr 0 & \varphi_1(z)-\varphi_0(z) \end{bmatrix}$.
\end{corollary}

\begin{proof}
By the proof of Lemma \ref{Diag},   $v_0=\frac{1}{\sqrt{2}}(1, \mu_{2}^{0})^{T}$ and $v_1=\frac{1}{\sqrt{2}}(1, \mu_{2}^{1})^{T}$ where $\mu_2^{1}=e^{\frac{2\pi i}{2}}=cos(\pi)+isin(\pi)=-1$. Then $$U=(v_0,v_1)=\frac{1}{\sqrt{2}}\begin{bmatrix} 1 & 1  \cr \mu_2^{0} & \mu_2^{1} \end{bmatrix}=\frac{1}{\sqrt{2}}\begin{bmatrix} 1 & 1  \cr 1 & -1 \end{bmatrix}$$
Thus
\begin{eqnarray}\label{n=2}
U^{*}\Phi(z)U&=&\frac{1}{\sqrt{2}}\begin{bmatrix} 1 & 1 \cr 1 & -1 \end{bmatrix} \begin{bmatrix} \varphi_0(z) & \varphi_1(z) \cr \varphi_1(z) & \varphi_0(z) \end{bmatrix}\cdot \frac{1}{\sqrt{2}}\begin{bmatrix} 1 & 1 \cr 1 & -1 \end{bmatrix}\cr
&=& \begin{bmatrix} \varphi_0(z)+\varphi_1(z) & 0 \cr 0 & \varphi_1(z)-\varphi_0(z) \end{bmatrix}\cr
&=&\begin{bmatrix} \lambda_0(z) & 0 \cr 0 & \lambda_1(z) \end{bmatrix}=\Lambda(z)
\end{eqnarray} where $\lambda_0(z)=\varphi_0(z)+\varphi_1(z) $ and $\lambda_1(z)=\varphi_1(z)-\varphi_0(z)$.
Then $\Phi$ is unitary equivalent to a diagonal matrix $\Lambda$.
Hence  $T_{\Phi}$ is unitarily equivalent to $T_{\Lambda}$ from Theorem \ref{BTOP}.
\end{proof}

\medskip

\begin{corollary}
Let $\Phi(z)=circ(\varphi_0(z), \varphi_1(z), \varphi_2(z))$. Then $T_{\Phi}$ is unitarily equivalent to $T_{\Lambda}$ where $\Lambda(z)=\begin{bmatrix} \lambda_0(z) & 0 & 0\cr 0 & \lambda_1(z) &0 \cr  0 & 0 & \lambda_2(z) \end{bmatrix}$ for $$\begin{cases}
\lambda_0(z)=\varphi_0(z)+\varphi_1(z)+\varphi_2(z), \cr
\lambda_1(z)=\mu_3\varphi_1(z)+\varphi_0(z)+\bar{\mu_3}\varphi_2(z),  \cr
\lambda_2(z)=\mu_3\varphi_2(z)+\varphi_0(z)+\bar{\mu_3}\varphi_1(z)
\end{cases}
$$ and $(\mu_3)^3=1$ and  $$\mu_3=e^{\frac{2\pi i}{3}}=cos(\frac{2\pi}{3})+isin(\frac{2\pi}{3})=\frac{-1+i\sqrt{3}}{2}.$$
\end{corollary}

\begin{proof}
By the proof of Lemma \ref{Diag},  $v_0=\frac{1}{\sqrt{3}}(1, \mu_{3}^{0}, \mu_{3}^{0})^{T}$, $v_1=\frac{1}{\sqrt{3}}(1, \mu_{3}^{1}, \mu_{3}^{2})^{T}$, and $v_2=\frac{1}{\sqrt{3}}(1, \mu_{3}^{2}, \mu_{3}^{4})^{T}=\frac{1}{\sqrt{3}}(1, \mu_{3}^{2}, \mu_{3}^{1})^{T}$  where $(\mu_3)^3=1$ and  $$\mu_3=e^{\frac{2\pi i}{3}}=cos(\frac{2\pi}{3})+isin(\frac{2\pi}{3})=\frac{-1+i\sqrt{3}}{2}.$$
Then $$U=(v_0,v_1,v_2)=\frac{1}{\sqrt{3}}\begin{bmatrix} 1 & 1& 1   \cr 1 & \mu_3^{1} &\mu_3^{2} \cr
1 & \mu_3^{2} &\mu_3^{4} \end{bmatrix}=\frac{1}{\sqrt{3}}\begin{bmatrix} 1 & 1& 1   \cr 1 & \mu_3 &\mu_3^{2} \cr
1 & \mu_3^{2} &\mu_3 \end{bmatrix}.$$
Since $\mu_3+\mu_3^2+1=0$ and $\bar{\mu_3}+\bar{\mu_3}^2+1=0$, it follows that
\begin{eqnarray*}\label{n=2}
&&U^{*}\Phi(z) U\cr
&=&\frac{1}{\sqrt{3}}\begin{bmatrix} 1 & 1& 1   \cr 1 & \bar{\mu_3} &\bar{\mu_3}^{2} \cr
1 & \bar{\mu_3}^{2} &\bar{\mu_3} \end{bmatrix} \begin{bmatrix} \varphi_0(z) & \varphi_1(z) &\varphi_2(z) \cr \varphi_2(z) & \varphi_0(z) & \varphi_1(z) \cr \varphi_1(z) & \varphi_2(z) & \varphi_0(z)\end{bmatrix}\cdot \frac{1}{\sqrt{3}}\begin{bmatrix} 1 & 1& 1   \cr 1 & \mu_3 &\mu_3^{2} \cr
1 & \mu_3^{2} &\mu_3 \end{bmatrix}\cr
&=& \small{\begin{bmatrix} \varphi_0(z)+\varphi_1(z)+\varphi_2(z) & 0 & 0 \cr 0 & \mu_3\varphi_1(z)+\varphi_0(z)+\bar{\mu_3}\varphi_2(z) & 0\cr 0 & 0 & \mu_3\varphi_2(z)+\varphi_0(z)+\bar{\mu_3}\varphi_1(z) \end{bmatrix}}\cr
&=&\begin{bmatrix} \lambda_0(z) & 0 &  \cr 0 & \lambda_1(z) & 0\cr 0 & 0 & \lambda_2(z) \end{bmatrix}=\Lambda (z)
\end{eqnarray*} where
$$\begin{cases}
\lambda_0(z)=\varphi_0(z)+\varphi_1(z)+\varphi_2(z), \cr
\lambda_1(z)=\mu_3\varphi_1(z)+\varphi_0(z)+\bar{\mu_3}\varphi_2(z),  \cr
\lambda_2(z)=\mu_3\varphi_2(z)+\varphi_0(z)+\bar{\mu_3}\varphi_1(z).
\end{cases}
$$
Then $\Phi$ is unitarily equivalent to a diagonal matrix $\Lambda$.
Hence  $T_{\Phi}$ is unitarily equivalent to $T_{\Lambda}$ from Theorem \ref{BTOP}.
\end{proof}

\medskip
\begin{corollary}\label{lambdais}
Let $\Phi\in L^{\infty}(\mathcal{C}_{n})$. Then $T_{\Phi}$ is binormal if and only if $T_{\lambda_{0}}, T_{\lambda_{1}},\cdots,T_{\lambda_{n-1}}$ are binormal, where $$U^{*}\Phi U=\Lambda=diag(\lambda_{0}, \lambda_{1},\cdots,\lambda_{n-1}).$$
\end{corollary}
\begin{proof}
The proof  follows from Theorem \ref{BTOP}.
\end{proof}

\medskip

Binormal Toeplitz operators on the classical Hardy space $H^2$ is characterized in \cite{GuLee}.
Let $\varphi\in L^{\infty}(\mathbb{T})$, and let $S$ be the unilateral shift on $H^{2}$. Set $A=T_{\varphi}^{*}T_{\varphi}$, $B=T_{\varphi}T_{\varphi}^{*}$, and  $F=S^{*}ABS-AB$.
\begin{lemma}{\em \cite[Lemma 2.1]{GuLee}}\label{guko}
$T_{\varphi}$ is binormal if and only if $F^{*}=F$.
\end{lemma}
\smallskip

\begin{corollary}
Let $\Phi\in L^{\infty}(\mathcal{C}_{n})$, $A_{j}=T_{\lambda_{j}}^{*}T_{\lambda_{j}}$, $B_{j}=T_{\lambda_{j}}T_{\lambda_{j}}^{*}$ for $j=0,1,2,\cdots,n-1$. Set $F_{j}=S^{*}A_{j}B_{j}S-A_{j}B_{j}$. Then $T_{\Phi}$ is binormal if and only if $F_{j}^{*}=F_{j}$.
\end{corollary}

\begin{proof}
By using Corollary \ref{lambdais} and Lemma \ref{guko}, the required result follows.
\end{proof}

\medskip
\begin{corollary}\label{cor} Let $\Phi\in L^{\infty}(\mathcal{C}_{n})$.
Then the following statements hold.\\
{\em (i)} Let ${\lambda_{k}}$ be analytic for every $k=0,1,2\cdots, {n-1}$. Then  ${\lambda_{k}}$  is constant multiple of an inner function for each $k$ if and only if $T_{\Phi}$ is binormal.\\
{\em (ii)} Let ${\lambda_{k}}$ be coanalytic for every $k=0,1,2\cdots, {n-1}$. Then  $\overline{\lambda_{k}}$  is constant multiple of an inner function for each $k$ if and only if  $T_{\Phi}$ is binormal.\\
{\em (iii)} Let ${\lambda_{k}}$ be a {\em(}neither analytic nor coanalytic{\em)} trigonometric poly
normal for all $k$. Then  $T_{\lambda_{k}}$ is normal if and only if $T_{\Phi}$ is binormal.
\end{corollary}

\begin{proof}
(i) Let ${\lambda_{k}}$ be analytic for every $k=0,1,2\cdots, {n-1}$. Then ${\lambda_{k}}$  is constant multiple of an inner function for each $k$ if and only if $T_{\lambda_{k}}$ is binormal (for all $k$) from \cite[Theorem 3.1]{GuLee}. Hence   $T_{\lambda_{k}}$ is binormal (for all $k$) if and only if $T_{\Phi}$ is binormal by Corollary \ref{lambdais}.

(ii) The proof follows from a similar way of (i).

(ii) Since  ${\lambda_{k}}$ is a (neither analytic nor coanalytic) trigonometric poly
normal, we conclude that $T_{\lambda_{k}}$ is normal if and only if  $T_{\lambda_{k}}$ is normal by Theorem 4.1  in \cite{GuLee}.
Hence $T_{\lambda_{k}}$ is binormal if and only if $T_{\lambda_k}$ is normal by Corollary \ref{lambdais}.
\end{proof}
\medskip

Even if $\Phi$ is normal, then $T_{\Phi}$ may not  be binormal, in general. In 1976, Abrahamese \cite{Ab} proved that if $\varphi$ is not analytic and $T_{\varphi}$ is hyponormal, then $\varphi$ is of bounded type if and only if $\overline{\varphi}$ is of bounded type.

\begin{example}
{\em (i)}
Let $\psi\in H^{\infty}$ be such that $\overline{\psi}$ is not of bounded type and set
$\Phi=\begin{pmatrix} z+\overline{z} & 0 \cr 0 & \psi \end{pmatrix}$.
Then it is clear that $\Phi$ is normal  and so  binormal. Moreover, $T_{\Phi}$ is hyponormal by \cite[Theorem 3.3]{Gu}.
Furthermore, $T_{\Phi}$ may not be binormal, in this case, the assumptions of Corollary \ref{cor} do not hold.

{\em (ii)} Let $\Phi(z)=\begin{pmatrix} 2 & 2\cr 2 & 2 \end{pmatrix}\overline{z}^2+\begin{pmatrix} 1 & 1\cr 1 & 1 \end{pmatrix}\overline{z}+\begin{pmatrix} 2\sqrt{2} & 2\sqrt{2}\cr 2\sqrt{2} & 2\sqrt{2} \end{pmatrix}{z}^2$. Then $\Phi$ is normal and so  binormal.
Therefore, $T_{\Phi}$ is hyponormal from \cite[Example 3.4]{HK}.
Moreover,  $T_{\Phi}$ may not be binormal, in this case, the assumptions of Corollary \ref{cor} does not hold.
\end{example}

\medskip

\begin{lemma}\label{reducing}
If $S, T\in \mathcal{L}(E)$ satisfy $T=USU^{*}$ for some unitary operator $U$, and if $S$ has a non-trivial closed reducing subspace, then
$T$ must also have a non-trivial closed reducing subspace, given by the image of the original reducing subspace under the unitary transformation.
\end{lemma}

\begin{proof}
Suppose that $S$ has a non-trivial closed reducing subspace $\mathcal{M}$, meaning that $S\mathcal{M}\subset\mathcal{M}$ and $S^{*}\mathcal{M}\subset \mathcal{M}$. Define the subspace $U\mathcal{M}=\mathcal{N}$. Since $U$ is unitary, $\mathcal{N}$ is also a non-trivial closed subspace of $E$. Moreover, since $\mathcal{M}$ reduces $S$, it follows that $S\mathcal{M}\subset \mathcal{M}$. Applying $U$, we obtain
$$US\mathcal{M}\subset U\mathcal{M}=\mathcal{N}.$$
Since $T=USU^{*}$, it follows that
$$T\mathcal{N}=USU^{*}\mathcal{N}=US\mathcal{M}\subset U\mathcal{M}=\mathcal{N}.$$
Thus, $\mathcal{N}$ is invariant under $T$.

Similarly, for the adjoint, using $S^{*}\mathcal{M}\subset \mathcal{M}$ and $T^{*}=US^{*}U^{*}$, we have
$$T^{*}\mathcal{N}=US^{*}U^{*}\mathcal{N}=US^{*}\mathcal{M}\subset U\mathcal{M}=\mathcal{N}. $$
Therefore, $\mathcal{N}$ is also invariant under $T^{*}$, confirming it a reducing subspace for $T$.
\end{proof}

The following proposition shows that the invariant subspace problem holds in this case.
\begin{proposition}
Let $\Phi=\begin{pmatrix}
\varphi_{0}&\varphi_{1}\\
\varphi_{1}&\varphi_{0}
\end{pmatrix}
\in L^{\infty}(\mathcal{C}_{2})$. Then $T_{\Phi}$ has a non-trivial closed reducing subspace.
\end{proposition}
\begin{proof}
Since $\Phi=\begin{pmatrix}
\varphi_{0}&\varphi_{1}\\
\varphi_{1}&\varphi_{0}
\end{pmatrix}
\in L^{\infty}(\mathcal{C}_{2})$, it follows from Lemma \ref{UE} that there exists a unitary operator $U$ and a diagonal function $\Lambda(z)$ such that
$$U^{*}\Phi(z)U=\begin{bmatrix}
\lambda_{0}(z) & 0 \\
0 & \lambda_{1}(z)
\end{bmatrix}
=\Lambda(z).
$$
Then Toeplitz operator corresponding to $\Lambda(z)$ is represented as
$$
T_{\Lambda}=\begin{bmatrix}
T_{\lambda_{0}} & 0 \\
0 & T_{\lambda_{1}}
\end{bmatrix}=T_{\lambda_{0}}\oplus T_{\lambda_{1}}.
$$
From the block diagonal representation of $T_{\Lambda}$, it follows that $T_{\Lambda}$ has a non-trivial closed reducing subspace. Since $T_{\Phi}$ is unitarily equivalent to $T_{\Lambda}$ by Theorem \ref{BTOP}, it follows from Lemma \ref{reducing} that $T_{\Phi}$ has a non-trivial closed reducing subspace.
\end{proof}

\medskip

\section{\bf{$\Gamma$-dilation of Toeplitz operators}}

Let  $\mathcal{C}_{n}$, $\mathcal{T}_{n}$, and $M_{n}$ be the spaces of matrices that are defined above. The operator $\Gamma: M_{n}\to \mathcal{C}_{n^{2}}$
 defined by $$\Gamma(A)=\Gamma([a_{ij}]_{i,j=0}^{n-1})={circ(a_{00}, a_{01},\cdots,a_{0n},\cdots, a_{(n-1)^{2}})}$$
is linear. Since $\text{dim} M_{n}=\text{dim}~\mathcal{C}_{n^{2}}$, it follows that $\Gamma$ is bijective.\\
If $\Phi=\begin{pmatrix}
\varphi_{ij}
\end{pmatrix}_{i,j=0}^{n-1}
\in L^{\infty}(M_{n})$, then $$\Gamma \Phi={circ(\varphi_{00}, \varphi_{01},\cdots,\varphi_{0n},\cdots, \varphi_{(n-1)^{2}}})\in L^{\infty}(\mathcal{C}_{n^{2}}).$$
If we set $dim E=n<\infty$, then
\begin{align*}
\Phi
=\begin{bmatrix}
\varphi_{11}  &\cdots & \varphi_{1n} \\
\vdots & \ddots  & \vdots \\
\varphi_{n1}  &\cdots & \varphi_{nn}
\end{bmatrix}
\end{align*}
and
\begin{align*}
T_{\Phi}
=\begin{bmatrix}
T_{\varphi_{11}}  &\cdots & T_{\varphi_{1n}} \\
\vdots & \ddots  & \vdots \\
T_{\varphi_{n1}}  &\cdots & T_{\varphi_{nn}}
\end{bmatrix}.
\end{align*}

Let $\mathcal{T}(H^{2}(E))$ and $\mathcal{T}(H^{2}(F))$ be the spaces of bounded Toeplitz operators on $H^{2}(E)$ and $H^{2}(F)$, respectively, where dim$E=n$ and dim$F=n^{2}$.
Then the operator ${\bf{\Gamma}}:\mathcal{T}(H^{2}(E))\to \mathcal{T}(H^{2}(F))$ defined by
$$
{\bf{\Gamma}}(T_{\Phi})=T_{\Gamma \Phi}
$$
is linear and bijective, where $T_{\Phi}\in \mathcal{T}(H^{2}(E))$ and $T_{\Gamma\Phi}\in \mathcal{T}(H^{2}(F))$. The Toeplitz operator $T_{\Gamma\Phi}$
 is called ${\bf{\Gamma}}$-\emph{dilation} of the Toeplitz operator $T_{\Phi}$, and $\Gamma \Phi$ is called the $\Gamma$-\emph{dilated} symbol.

In this section, we discuss the application of the ${\bf{\Gamma}}$-\emph{dilated} Toeplitz operators $T_{\Gamma\Phi}$.
The adjoint of $\Gamma$ is $\Gamma^{*}:\mathcal{C}_{n^{2}}\to M_{n}$, and is given by the formula
$$\Gamma^{*}(C)=\Gamma^{*}(circ(a_{11}, a_{12},\cdots,a_{1n},\cdots,a_{n^{2}}))=\begin{bmatrix}
n^{2}a_{11} &\cdots & n^{2}a_{1n} \\
\vdots &\ddots & \vdots \\
n^{2}a_{n1} &\cdots & n^{2}a_{nn}
\end{bmatrix}$$
where $C\in\mathcal{C}_{n^{2}}$.
Since $\Gamma \Phi\in L^{\infty}(\mathcal{C}_{n^{2}})$, then by Lemma \ref{UE}, $\Gamma \Phi$ is unitarily equivalent to the diagonal matrix $\Lambda$, i.e.,
$$U^{*}\Gamma \Phi U=\Lambda=diag(\lambda_{0}, \lambda_{1},\cdots,\lambda_{n-1},\cdots,\lambda_{(n-1)^{2}})$$

\medskip

Let $\Phi\in L^{\infty}(M_{2})$, i.e.,
\begin{align*}
\Phi
=\begin{bmatrix}
\varphi_{0}  & \varphi_{1} \\
\varphi_{2}  & \varphi_{3}
\end{bmatrix}.
\end{align*}
Then
\begin{align*}
\Gamma\Phi
=\begin{bmatrix}
\varphi_{0}  &\varphi_{1} &\varphi_{2}& \varphi_{3} \\
\varphi_{3}  &\varphi_{0} &\varphi_{1}& \varphi_{2}\\
\varphi_{2}  &\varphi_{3} &\varphi_{0}& \varphi_{1}\\
\varphi_{1}  &\varphi_{2} &\varphi_{3}& \varphi_{0}
\end{bmatrix}
=\begin{bmatrix}
\Psi_{11}  & \Psi_{22} \\
\Psi_{22}  & \Psi_{11}
\end{bmatrix}\in L^{\infty}(\mathcal{C}_{4}),
\end{align*}
where
\begin{equation}\label{EqPsi}
\Psi_{11}
=\begin{bmatrix}
\varphi_{0}  & \varphi_{1} \\
\varphi_{3}  &\varphi_{0}
\end{bmatrix}\quad
\text{and}\quad
\Psi_{22}
=\begin{bmatrix}
\varphi_{2}  & \varphi_{3} \\
\varphi_{1}  & \varphi_{2}
\end{bmatrix}.
\end{equation}
It is clear that $\Psi_{11}$ and $\Psi_{22}$ are not circulant matrices but are Toeplitz matrices.

Since $\Gamma \Phi\in L^{\infty}(\mathcal{C}_{4})$, it is unitary equivalent to

\begin{align*}
\Lambda
=\begin{bmatrix}
\lambda_{0}  &0 &0& 0 \\
0  &\lambda_{1} &0& 0\\
0  &0 &\lambda_{2}& 0\\
0  &0 &0& \lambda_{3}
\end{bmatrix}
=\begin{bmatrix}
\Lambda_{11}  & \bf{0} \\
\bf{0}  & \Lambda_{22}
\end{bmatrix}
\end{align*}
where
\begin{align*}
\Lambda_{11}
=\begin{bmatrix}
\lambda_{0}  & 0 \\
0  & \lambda_{1}
\end{bmatrix}\quad
\text{and} \quad
\Lambda_{22}
=\begin{bmatrix}
\lambda_{2}  & 0 \\
0  & \lambda_{3}
\end{bmatrix}.
\end{align*}
The following theorem shows the relation between the Toeplitz operators
$T_{\Psi_{ii}}$ and $T_{\Lambda_{ii}}$ for $i=1,2$. Moreover, this theorem is about the invariant subspace of the block Toeplitz operator with a matrix-valued symbol.

\begin{theorem}\label{TU}
Let  $\Phi
=\begin{bmatrix}
\varphi_{0}  & \varphi_{1} \\
\varphi_{2}  &\varphi_{3}
\end{bmatrix}\in L^{\infty}(M_2)
$ and $\Psi_{11}
=\begin{bmatrix}
\varphi_{0}  & \varphi_{1} \\
\varphi_{3}  &\varphi_{0}
\end{bmatrix}
$ be diagonal components of $\Gamma \Phi$ which is the $\Gamma$-dilated symbol. Then the following statements hold.\\
{\em (i)} The Toeplitz operator $T_{\Psi_{11}}$ is unitary equivalent to $T_{\Lambda_{11}}$.\\
{\em (ii)} The Toeplitz operator $T_{\Psi_{11}}$ has a non-trivial closed reducing subspace.
\end{theorem}

\begin{proof} (i)
Since $\Psi_{11}\in L^{\infty}(\mathcal{T}_{2})$, it follows that $T_{\Psi_{11}}$ is the compression of $\Gamma$-dilated Toeplitz operator $T_{\Gamma\Phi}$, i.e.,
$$T_{\Psi_{11}}=P_{H^{2}(\mathbb{C}^{2})}T_{\Gamma\Phi}P_{H^{2}(\mathbb{C}^{2})}.$$
By Theorem \ref{BTOP}, $T_{\Gamma\Phi}$ is unitarily equivalent to $T_{\Lambda}$. Therefore
\begin{align*}
T_{\Psi_{11}}
&=P_{H^{2}(\mathbb{C}^{2})}T_{\Gamma\Phi}P_{H^{2}(\mathbb{C}^{2})}\\
&=P_{H^{2}(\mathbb{C}^{2})}UT_{\Lambda}U^{*}P_{H^{2}(\mathbb{C}^{2})}\\
&=UP_{H^{2}(\mathbb{C}^{2})}T_{\Lambda}P_{H^{2}(\mathbb{C}^{2})}U^{*}\\
&=UT_{\Lambda_{11}}U^{*}.
\end{align*}

(ii)
By (i), the Toeplitz operator $T_{\Psi_{11}}$ is unitarily equivalent to the Toeplitz operator $T_{\Lambda}$. But $T_{\Lambda}$ has a $2\times 2$ block diagonal representation, i.e.,
$$T_{\Lambda}=\begin{bmatrix}
T_{\lambda_{0}}  & 0 \\
0  & T_{\lambda_{1}}
\end{bmatrix}=T_{\lambda_{0}}\oplus  T_{\lambda_{1}}.$$
From the block representation of $T_{\Lambda}$, it follows that $T_{\Lambda}$ has a non-trivial closed reducing subspace. Hence by Lemma \ref{reducing}, $T_{\Psi_{11}}$ has non-trivial closed reducing subspace.
\end{proof}

\begin{corollary}
Let $\Psi_{11}=\begin{bmatrix}
\varphi_{0}  &\varphi_{1}  \\
\varphi_{3}  & \varphi_{0}
\end{bmatrix}\in L^{\infty}(M_{2})$. The Toeplitz operator $T_{\Psi_{11}}$ is binormal if and only if $T_{\Lambda_{11}}$ is binormal.
\end{corollary}
\begin{proof}
The proof follows from Lemma \ref{UE} and Theorem \ref{TU}.
\end{proof}

\section{\bf{Binormal Toeplitz operators with matrix valued symbols}}

\medskip

In this section, we study binormal Toepltiz operators with matrix valued symbols.
The classical normal Toeplitz operators were characterized by Brown and Halmos in \cite{BH}. They proved that $T_{\varphi}$ is normal if and only if $\varphi=\alpha f+ \beta$ for some real $\alpha, \beta\in \mathbb{C}$ and $f\in L^{\infty}$ is a real valued function.  It is well known that if $\psi$ is analytic, then  $T_{\varphi}T_{\psi}=T_{\varphi\psi}$ and $T_{\overline{\psi}}T_{\varphi}=T_{\overline{\psi}\varphi}$.
The Fuglede-Putnam theorem  says that if $N$ is normal and $X$ is any operator with $NX=XN$, then $N^{\ast}X=XN^{\ast}$ holds.

\medskip

Let $T$ be $2$-normal, i.e., $T$ is unitarily equivalent to an operator of the form
$\begin{pmatrix}T_1 & T_2 \\ T_3 & T_4 \end{pmatrix}\in{\mathcal L}({\mathcal H} \oplus {\mathcal H})$ where $T_i$ are commuting normal operators for $i=1,2,3,4$. Then it is well known from {\cite[Theorem 1]{GW}} that $T$ is complex symmetric.
Also, $T$ is $2$-normal if and only if T is unitarily  equivalent to an upper triangular operator matrix.
\begin{proposition}\label{ex}
Let $\Phi=\begin{pmatrix}
\varphi_{1}&\varphi_{2}\\
\varphi_{3}&\varphi_{4}
\end{pmatrix}
\in L^{\infty}(M_{2})$ and $T_{\Phi}=\begin{pmatrix}
T_{\varphi_{1}}&T_{\varphi_{2}}\\
T_{\varphi_{3}}&T_{\varphi_{4}}
\end{pmatrix}$ such that $T_{\varphi_{i}}$ are mutually commuting normal operators. Let
\[
\left\{
\begin{array}{l}
\text{$t_{1}=T_{\varphi_{1}}^{*}T_{\varphi_{1}}+T_{\varphi_{3}}^{*}T_{\varphi_{3}}$}\\
\text{$t_{2}=T_{\varphi_{1}}^{*}T_{\varphi_{2}}+T_{\varphi_{3}}^{*}T_{\varphi_{4}}$} \\
\text{$t_{3}=T_{\varphi_{2}}^{*}T_{\varphi_{2}}+T_{\varphi_{4}}^{*}T_{\varphi_{4}}$}\\
\text{$s_{1}=T_{\varphi_{1}}T_{\varphi_{1}}^{*}+T_{\varphi_{2}}T_{\varphi_{2}}^{*}$} \\
\text{$s_{2}=T_{\varphi_{1}}T_{\varphi_{3}}^{*}+T_{\varphi_{2}}T_{\varphi_{4}}^{*}$}\\
\text{$s_{3}=T_{\varphi_{3}}T_{\varphi_{3}}^{*}+T_{\varphi_{4}}T_{\varphi_{4}}^{*}$}.
\end{array}
\right.
\]
 Then the following statements hold.\\
{\em (i)} $T_{\Phi}$ is binormal if and only if
 \begin{equation}\label{E_00}
\left\{
\begin{array}{l}
\text{$(s_{2}t_{2}^{*})^{*}=s_{2}t_{2}^{*}$}\\
\text{$(s_{2}^{*}t_{2})^{*}=s_{2}^{*}t_{2}$} \\
\text{$t_{1}s_{2}+t_{2}s_{3}=s_{1}t_{2}+s_{2}t_{3}$}.
\end{array}
\right.
\end{equation}
{\em (ii)} { $T_{\Phi}$ is  normal if and only if $T_{\varphi_{3}}^{*}T_{\varphi_{3}}=T_{\varphi_{2}}T_{\varphi_{2}}^{*}, T_{\varphi_{2}}^{*}T_{\varphi_{2}}=T_{\varphi_{3}}T_{\varphi_{3}}^{*}$ and $T_{\varphi_{1}}^{*}T_{\varphi_{2}}+T_{\varphi_{3}}^{*}T_{\varphi_{4}}=T_{\varphi_{1}}T_{\varphi_{3}}^{*}+T_{\varphi_{2}}T_{\varphi_{4}}^{*}$.}

\end{proposition}

\begin{proof}
(i) By \cite[Theorem 2.1]{KKJEL},  $T_{\Phi}$ is binormal if and only if
\begin{eqnarray} \label{ST}
\left\{
\begin{array}{l}
\text{$t_{1}s_{1}+t_{2}s_{2}^{*}=s_{1}t_{1}+s_{2}t_{2}^{*}$}\\
\text{$t_{3}s_{3}+t_{2}^{*}s_{2}=s_{3}t_{3}+s_{2}^{*}t_{2}$} \\
\text{$t_{1}s_{2}+t_{2}s_{3}=s_{1}t_{2}+s_{2}t_{3}$}.
\end{array}
\right.
\end{eqnarray}
Since it is given that $T_{\varphi_{i}}$ are mutually commuting normal operators then by Fuglede-Putnam theorem,
$T_{\varphi_{i}}^{*}T_{\varphi_{j}}=T_{\varphi_{j}}T_{\varphi_{i}}^{*}$ for $i,j=1,2,3$.
From this and $T_{\varphi_i}$ is normal for $i=1,2,3,4$, we have
\begin{eqnarray*}
t_1s_1-s_1t_1&=&(T_{\varphi_{1}}^{*}T_{\varphi_{1}}+T_{\varphi_{3}}^{*}T_{\varphi_{3}})(T_{\varphi_{1}}T_{\varphi_{1}}^{*}+T_{\varphi_{2}}T_{\varphi_{2}}^{*})\cr
&&-(T_{\varphi_{1}}T_{\varphi_{1}}^{*}+T_{\varphi_{2}}T_{\varphi_{2}}^{*})(T_{\varphi_{1}}^{*}T_{\varphi_{1}}+T_{\varphi_{3}}^{*}T_{\varphi_{3}})\cr
&=&T_{\varphi_{1}}^{*}T_{\varphi_{1}}T_{\varphi_{2}}T_{\varphi_{2}}^{*}+T_{\varphi_{3}}^{*}T_{\varphi_{3}}T_{\varphi_{1}}^{*}T_{\varphi_{1}}+T_{\varphi_{3}}^{*}T_{\varphi_{3}}T_{\varphi_{2}}^{*}T_{\varphi_{2}}\cr
&&-T_{\varphi_{1}}T_{\varphi_{1}}^{*}T_{\varphi_{3}}^{*}T_{\varphi_{3}}-T_{\varphi_{2}}T_{\varphi_{2}}^{*}T_{\varphi_{1}}^{*}T_{\varphi_{1}}-T_{\varphi_{2}}T_{\varphi_{2}}^{*}T_{\varphi_{3}}^{*}T_{\varphi_{3}}=0
\end{eqnarray*}
and by  a similar way, we show that $t_{3}s_{3}=s_{3}t_{3}$.
 Therefore $t_{i}s_{i}=s_{i}t_{i}$ for $i=1,3$.
Hence \eqref{ST} becomes
\begin{eqnarray*}
\left\{
\begin{array}{l}
\text{$(s_{2}t_{2}^{*})^{*}=s_{2}t_{2}^{*}$}\\
\text{$(s_{2}^{*}t_{2})^{*}=s_{2}^{*}t_{2}$} \\
\text{$t_{1}s_{2}+t_{2}s_{3}=s_{1}t_{2}+s_{2}t_{3}$}.
\end{array}
\right.
\end{eqnarray*}
(ii) {Since  $T_{\varphi_{i}}$ are normal, we conclude that $T_{\Phi}$ is  normal if and only if $T_{\varphi_{3}}^{*}T_{\varphi_{3}}=T_{\varphi_{2}}T_{\varphi_{2}}^{*}, T_{\varphi_{2}}^{*}T_{\varphi_{2}}=T_{\varphi_{3}}T_{\varphi_{3}}^{*}$ and $T_{\varphi_{1}}^{*}T_{\varphi_{2}}+T_{\varphi_{3}}^{*}T_{\varphi_{4}}=T_{\varphi_{1}}T_{\varphi_{3}}^{*}+T_{\varphi_{2}}T_{\varphi_{4}}^{*}$.}
\end{proof}
\smallskip

\begin{corollary}
Let $\Phi=\begin{pmatrix}
\varphi_{1}&\varphi_{2}\\
\varphi_{3}&\varphi_{4}
\end{pmatrix}
\in L^{\infty}(M_{2})$ and $T_{\Phi}=\begin{pmatrix}
T_{\varphi_{1}}&T_{\varphi_{2}}\\
T_{\varphi_{3}}&T_{\varphi_{4}}
\end{pmatrix}$ such that $T_{\varphi_{i}}$ are mutually commuting normal operators.
Then the following statements hold.\\
{\em (i)} If $\varphi_1=\varphi_4=I$, then  $T_{\Phi}$ is binormal if and only if $$(T_{\varphi_{3}}^{*}T_{\varphi_{3}}-T_{\varphi_{2}}^{*}T_{\varphi_{2}})(T_{\varphi_{2}}+T_{\varphi_{3}}^{*})+(T_{\varphi_{2}}+T_{\varphi_{3}}^{*})(T_{\varphi_{3}}^{*}T_{\varphi_{3}}-T_{\varphi_{2}}^{*}T_{\varphi_{2}})=0.$$
{\em (ii)} If $\varphi_2=\varphi_3=I$, then  $T_{\Phi}$ is binormal if and only if  $$t_{1}(t_{2}^{\ast}-t_{2})+(t_{2}-t_{2}^{\ast})t_{3}=0~\mbox{and}~s_2^{2\ast}=s_2^2.$$
{\em (iii)} If $\varphi_1=\varphi_4=0$ or $\varphi_2=\varphi_3=0$, then  $T_{\Phi}$ is binormal.
\end{corollary}
\begin{proof}
(i) If $\varphi_1=\varphi_4=I$, then
$$
\begin{cases}
\text{$t_{1}=I+T_{\varphi_{3}}^{*}T_{\varphi_{3}}$}\\
\text{$t_{2}=T_{\varphi_{2}}+T_{\varphi_{3}}^{*}$} \\
\text{$t_{3}=T_{\varphi_{2}}^{*}T_{\varphi_{2}}+I$}\\
\text{$s_{1}=I+T_{\varphi_{2}}T_{\varphi_{2}}^{*}$} \\
\text{$s_{2}=T_{\varphi_{3}}^{*}+T_{\varphi_{2}}$}\\
\text{$s_{3}=T_{\varphi_{3}}T_{\varphi_{3}}^{*}+I$}.
\end{cases}
$$
Since $T_{\varphi_{i}}$ are mutually commuting normal operators, $t_1=s_3, t_2=s_2$, and $t_3=s_1.$
By Proposition \ref{ex}, $T_{\Phi}$ is binormal if and only if $t_{1}t_{2}+t_{2}t_{1}=t_{3}t_{2}+t_{2}t_{3}$ and it implies that $$(t_{1}-t_3)t_{2}+t_{2}(t_{1}-t_{3})=0.$$
Therefore,  $T_{\Phi}$ is binormal if and only if $$(T_{\varphi_{3}}^{*}T_{\varphi_{3}}-T_{\varphi_{2}}^{*}T_{\varphi_{2}})(T_{\varphi_{2}}+T_{\varphi_{3}}^{*})+(T_{\varphi_{2}}+T_{\varphi_{3}}^{*})(T_{\varphi_{3}}^{*}T_{\varphi_{3}}-T_{\varphi_{2}}^{*}T_{\varphi_{2}})=0.$$

(ii)  If $\varphi_2=\varphi_3=I$, then
$$
\begin{cases}
\text{$t_{1}=T_{\varphi_{1}}^{*}T_{\varphi_{1}}+1$}\\
\text{$t_{2}=T_{\varphi_{4}}+T_{\varphi_{1}}^{*}$} \\
\text{$t_{3}=T_{\varphi_{4}}^{*}T_{\varphi_{4}}+I$}\\
\text{$s_{1}=I+T_{\varphi_{1}}T_{\varphi_{1}}^{*}$} \\
\text{$s_{2}=T_{\varphi_{4}}^{*}+T_{\varphi_{1}}$}\\
\text{$s_{3}=T_{\varphi_{4}}T_{\varphi_{4}}^{*}+I$}.
\end{cases}
$$
Since $T_{\varphi_{i}}$ are mutually commuting normal operators, $t_1=s_1, t_2=(s_2)^{\ast}$, and $t_3=s_3.$
By Proposition \ref{ex}, $T_{\Phi}$ is binormal if and only if $t_{1}(t_{2}^{\ast}-t_{2})+(t_{2}-t_{2}^{\ast})t_{3}=0$ and $s_2^{2\ast}=s_2^2$.

(iii) If $\varphi_1=\varphi_4=0$ or  $\varphi_2=\varphi_3=0$, then $s_2=t_2=0$ and so (\ref{E_00}) holds.
Hence $T_{\Phi}$ is binormal from Proposition \ref{ex}.
\end{proof}

\medskip

{
\begin{corollary}
Let $\Phi=\begin{pmatrix}
\varphi_{1}&\varphi_{2}\\
\varphi_{2}&\varphi_{4}
\end{pmatrix}
\in L^{\infty}(M_{2})$. Then $T_{\Phi}=\begin{pmatrix}
T_{\varphi_{1}}&T_{\varphi_{2}}\\
T_{\varphi_{2}}&T_{\varphi_{4}}
\end{pmatrix}$ such that $T_{\varphi_{i}}$ are mutually commuting normal operators.
Then the folllowing statements hold.\\
{\em (i)} $T_{\Phi}$ is normal if and only if
\text{$T_{\varphi_{1}}^{*}T_{\varphi_{2}}+T_{\varphi_{2}}^{*}T_{\varphi_{4}}
=T_{\varphi_{1}}T_{\varphi_{2}}^{*}+T_{\varphi_{2}}T_{\varphi_{4}}^{*}.$}\\
{\em (ii)} If $\varphi_2=I$, then $T_{\Phi}$ is normal if and only if  ${\varphi_{1}+\overline{\varphi_{4}}}$ is a real-valued function.
\end{corollary}

\begin{proof}
(i) By Proposition \ref{ex}, $T_{\Phi}$ is normal if and only if
\begin{eqnarray*}
\left\{
\begin{array}{l}
\text{$T_{\varphi_{3}}^{*}T_{\varphi_{3}}
=T_{\varphi_{2}}T_{\varphi_{2}}^{*}$} \\
\text{$T_{\varphi_{1}}^{*}T_{\varphi_{2}}+T_{\varphi_{3}}^{*}T_{\varphi_{4}}
=T_{\varphi_{1}}T_{\varphi_{3}}^{*}+T_{\varphi_{2}}T_{\varphi_{4}}^{*}$} \\
\text{$T_{\varphi_{2}}^{*}T_{\varphi_{2}}=T_{\varphi_{3}}T_{\varphi_{3}}^{*}
$}.
\end{array}
\right.
\end{eqnarray*}
Since $T_{\varphi_{3}}=T_{\varphi_2}$,  we obtain that $T_{\Phi}$ is normal if and only if
$$T_{\varphi_{1}}^{*}T_{\varphi_{2}}+T_{\varphi_{2}}^{*}T_{\varphi_{4}}
=T_{\varphi_{1}}T_{\varphi_{2}}^{*}+T_{\varphi_{2}}T_{\varphi_{4}}^{*}.$$

(ii) If $\varphi_2=I$, then by (i), $$T_{\varphi_{1}}^{*}T_{\varphi_{2}}+T_{\varphi_{2}}^{*}T_{\varphi_{4}}
=T_{\varphi_{1}}T_{\varphi_{2}}^{*}+T_{\varphi_{2}}T_{\varphi_{4}}^{*}$$ becomes $T_{\varphi_{1}}^{*}+T_{\varphi_{4}}
=T_{\varphi_{1}}+T_{\varphi_{4}}^{*}$. Therefore $T_{\Phi}$ is normal if and only if  $T_{\varphi_{1}+\overline{\varphi_{4}}}$ is self-adjont.
\end{proof}
}
\medskip

A direct calculation shows that the following examples are binormal.
\begin{example}
{\em(a)} Let $\Phi=\begin{pmatrix}
0&0\\
\varphi&0
\end{pmatrix}
$. Then $T_{\Phi}=\begin{pmatrix}
0&0\\
T_{\varphi}&0
\end{pmatrix}$
is binormal but not normal.

{\em(b)}
Let $\Psi=\begin{pmatrix}
0&I\\
\psi&0
\end{pmatrix}
$. Then $T_{\Psi}=\begin{pmatrix}
0&I\\
T_{\psi}&0
\end{pmatrix}$
is binormal and  it is normal when $T_{\psi}$ is unitary.
\end{example}

\medskip

\noindent $\mathbf{Acknowledgment}$

\noindent The authors wish to thank the referees for their invaluable comments on the original draft. The first and last authors are supported by the project TUBITAK 1001, 123F356.

\medskip

\end{document}